\documentclass[a4paper,11pt]{amsart}

\usepackage[french, ngerman, italian, english]{babel}
\usepackage{amsmath,amsthm, amssymb}
\usepackage{setspace}
\usepackage{anysize}
\usepackage{url}
\usepackage{bm}
\usepackage{graphicx}
\usepackage{amscd}
\usepackage[colorlinks=true]{hyperref}
\usepackage{mathptmx}
\usepackage{paralist}
\usepackage{verbatim}
\usepackage[utf8]{inputenc}

\theoremstyle{plain}
\newtheorem{thm}{Theorem}[section]

\newtheorem{cor}[thm]{Corollary}
\newtheorem{lem}[thm]{Lemma}
\newtheorem{prop}[thm]{Proposition}

\theoremstyle{definition}

\newtheorem{os}[thm]{Remark}
\newtheorem{es}[thm]{Example}

\def \chara{{\operatorname{char}}}
\def \cd{{\operatorname{cd}}}

\def \height{{\operatorname{ht}}}

\def \depth{{\operatorname{depth}}}
\def \grade{{\operatorname{grade}}}
\def \Spec{{\operatorname{Spec}}}
\def \Proj{{\operatorname{Proj}}}

\def\Supp{\operatorname{Supp}}

\def\Ext{\operatorname{Ext}}

\def \mm{{\mathfrak{m}}}

\def \nn{{\mathfrak{n}}}

\def \kk{{\kappa}}

\def \PP{\mathbb P}

\def \CC{\mathbb C}

\def \V{\mathcal V}

\def \O{\mathcal O}
\def \F{\mathcal F}

\def \NN{{\mathbb{N}}}

\def \ZZ{{\mathbb{Z}}}
\def \QQ{{\mathbb{Q}}}

\def \aa{{\mathfrak{a}}}

\def \bb{{\mathfrak{b}}}
\def \nn{{\mathfrak{n}}}
\def \kk{{\Bbbk}}

\def \Ha{H_{\mathfrak{a}}}

\def \red{{\operatorname{red}}}

\def \Hom{{\operatorname{Hom}}}
\def \Ext{{\operatorname{Ext}}}

\begin{document}

\title{Cohomological and projective dimensions}
\author{Matteo Varbaro}
\address{Dipartimento di Matematica,
Universit\`a degli Studi di Genova, Italy}
\email{varbaro@dima.unige.it}
\date{}
\maketitle

\begin{abstract}
Let $\aa$ be a homogeneous ideal of a polynomial ring $R$ in $n$ variables over a field $\kk$. Assume that $\depth(R/\aa)\geq t$, where $t$ is some number in $\{0,\ldots ,n\}$. A result of Peskine and Szpiro says that, if $\chara(\kk)>0$, then the local cohomology modules $H_{\aa}^i(M)$ vanish for all $i>n-t$ and all $R$-modules $M$. In characteristic $0$, there are counterexamples to this for all $t\geq 4$. On the other hand, when $t\leq 2$, exploiting classical results of Grothendieck, Lichtenbaum, Hartshorne and Ogus it is not difficult to extend the result to any characteristic. In this paper we settle the remaining case, namely we show: If $\depth(R/\aa)\geq 3$, then the local cohomology modules $H_{\aa}^i(M)$ vanish for all $i>n-3$ and all $R$-modules $M$, whatever the characteristic of $\kk$ is.
\end{abstract}

\section{Introduction}

In his seminar on local cohomology \cite[p. 79]{gro1}, Grothendieck raised the problem of finding conditions under which, fixed a positive integer $c$, the local cohomology modules $\Ha^i(R)$ vanish for every $i>c$, where $\aa$ is an ideal in a ring $R$. In other words, looking for conditions under which the cohomological dimension $\cd(R,\aa)\leq c$. Ever since many mathematicians worked on this question (for instance see Hartshorne \cite{Ha}, Ogus \cite{Og}, Hartshorne and Speiser \cite{HS}, Faltings \cite{Fa}, Huneke and Lyubeznik \cite{HL}, Lyubeznik \cite{Ly1}). In this spirit, we will study the relationships between cohomological and projective dimensions. Before explaining the results of the paper, we wish to summarize the essential known facts about this subject.

The two starting results are both due to Grothendieck (see \cite{gro1}): They are essential, as they fix the range in which we must look for the natural number $\cd(R,\aa)$:
$$\height(\aa)\leq \cd(R,\aa)\leq \dim R.$$
Afterwards, first Lichtenbaum and then, more generally, Hartshorne \cite{Ha}, settled the problem of characterizing when $\cd(R,\aa)\leq \dim R-1$. Roughly speaking, they showed that the necessary and sufficient condition for this to happen is that $\dim R/\aa > 0$. By then, the next step should have been to describe when $\cd(R,\aa)\leq \dim R-2$. In general this case is still not understood. However, if $R$ is a complete regular local ring containing a field, a necessary and sufficient condition is that the punctured spectrum of $R/\aa$ is connected. This has been shown in \cite{HS} in positive characteristic and in \cite{Og} in characteristic $0$. In \cite{HL} a characteristic free proof is given. Actually, if the ambient ring $R$ is regular, $\cd(R,\aa)$ can always be characterized in terms of the ring $R/\aa$ (\cite{HS,Og,Ly1}). However, in any of these papers, the described conditions are quite difficult to verify.

A classical result of Peskine and Szpiro obtained in \cite{PS} says that, if $\aa$ is a perfect ideal of a regular local ring of characteristic $p>0$, then $\cd(R,\aa)=\height(\aa)$. Their proof is based on the flatness of the Frobenius map $R\rightarrow R$ which, by the work of Kunz \cite{Ku}, is equivalent to $R$ being regular. We will notice that their idea works for all Noetherian rings of positive characteristic, see Corollary \ref{corcharp} (indeed also in some situations in characteristic $0$, see Lemma \ref{generalthm}), exploiting the acyclicity criterion that Buchsbaum and Eisenbud got in \cite{BE} instead of the result of Kunz. 

In characteristic $0$ the situation is completely different. There are several instances of perfect ideals with high cohomological dimension, see Example \ref{exdet}: Such examples appear even in a regular ambient, thus we stick to the situation in which $R$ is an $n$-dimensional regular local ring containing a field. In such a circumstance, if $\aa$ is a perfect ideal of height $n-2$, then $\cd(R,\aa)=n-2$. More generally, if $\depth(R/\aa)\geq 2$, then $\cd(R,\aa)\leq n-2$, see Proposition \ref{dimension 2}. We know examples of perfect ideals of height $n-4$ and cohomological dimension $n-3$. However we do not know any example of ideals $\aa\subseteq R$ of projective dimension $\leq n-3$ such that $\cd(R,\aa)>n-3$. Therefore we consider the case in which $\depth(R/\aa)\geq 3$. We show in Proposition \ref{lyuthm} that $H_{\aa}^i(R)=0$ for all $i\geq n-1$ and $H_{\aa}^{n-2}(R)_{\mathfrak{p}}=0$ for any prime ideal of $R$ different from the maximal one. As a consequence, we get that the Lyubeznik numbers of a local ring $A$ (defined in \cite{Ly2}) $\lambda_{i,j}(A)$ vanish for all $0\leq j<\depth(A)$ and $i\geq j-1$, see Corollary \ref{lyucor}. At this point, we focus on the special case where $R$ is a polynomial ring in $n$ variables over a field of characteristic $0$ and $\aa$ is a homogeneous ideal such that $\depth(R/\aa)\geq 3$. The main result of the paper is that, under the above assumptions, $\cd(R,\aa)\leq n-3$ (see Theorem \ref{mainthm}). In particular, if $\aa$ is perfect of height $n-3$, then $\cd(R,\aa)=n-3$. As a consequence, in Remark \ref{setdepth} we get several examples of prime ideals $\aa\subseteq R$ such that $R/\aa$ is not set-theoretically Cohen-Macaulay, thereby generalizing a result of Singh and Walther in \cite{SW}. More such examples can be produced using Proposition \ref{hodge}. 

\vspace{2mm}

I am strongly indebted with Vasudevan Srinivas: Without his advice the results of this paper would be much weaker. I also wish to thank Lucian Badescu for some helpful discussions, the anonymous referee for reading the paper extremely carefully (giving me relevant suggestions) and Srikanth Iyengar for explaining me the process of {\it gonflement} of a local ring.

\section{Ideals with small cohomological dimension}

Given an ideal $\aa$ of a Noetherian ring $R$, its {\it cohomological dimension}, denoted $\cd(R,\aa)$, is the least natural number $c$ such that the local cohomology modules  $H_{\aa}^i(M)$ vanish for all $R$-modules $M$ and $i>c$. As it is well known, it is enough to check the condition for $M=R$. We have
\[\cd(R,\aa)\geq \height(\aa),\]
and in this section we will find some ideals for which the above inequality is an equality. It is convenient to remind a definition: An ideal $\aa\subseteq R$ is {\it perfect} if the maximal length of an $R$-regular sequence in $\aa$, namely $\grade(\aa)$, is equal to the projective dimension of $R/\aa$ (in particular a perfect ideal has finite projective dimension). If $R$ is a regular local ring, then $\aa$ is perfect if and only if $R/\aa$ is Cohen-Macaulay. The following fact has many interesting corollaries:

\begin{lem}\label{generalthm}
Let $R$ be a Noetherian ring and $\aa\subseteq R$ a perfect ideal. Assume that for all integers $k\in\NN$ there is a ring-homomorphism $\phi_k:R\rightarrow R$ such that:
\begin{enumerate}
\item[{\em 1}] $\phi_0=1_R$.
\item[{\em 2}] $\phi_i(\aa)\subseteq \phi_j(\aa)$ whenever $i\geq j$.
\item[{\em 3}] The inverse system of ideals $\{\phi_k(\aa)R\}_{k\in\NN}$ is cofinal with $\{\aa^k\}_{k\in\NN}$.
\end{enumerate}
Then $\cd(R,\aa)=\height(\aa)$.
\end{lem}
\begin{proof}
Set $\grade(\aa)=g$. Since for all $i$ there exists $j$ such that $\aa^j\subseteq \phi_i(\aa)R \subseteq \aa$, we have $\grade (\phi_k(\aa)R)=g$ for all $k\in\NN$. So $\phi_k(\aa)R$ is a perfect ideal for all $k\in\NN$ by \cite[Theorem 3.5]{BV}. Therefore 
\[\Ext_R^i(R/\phi_k(\aa)R,R)=0 \ \ \forall \ i>g, \ k\in\NN.\] 
We infer that $\cd(R,\aa)\leq g$ by the identity 
\[H_{\aa}^i(R)\cong \varinjlim \Ext^i(R/\phi_k(\aa),R).\]
We conclude because $\cd(R,\aa)\geq\height(\aa)\geq \grade(\aa)$.
\end{proof}

The first consequence of Lemma \ref{generalthm} is the promised extension of \cite[Proposition 4.1]{PS}.

\begin{cor}\label{corcharp}
Let $R$ be a Noetherian ring of positive characteristic. If $\aa\subseteq R$ is a perfect ideal, then $\cd(R,\aa)=\height(\aa)$.
\end{cor}
\begin{proof}
It follows from Lemma \ref{generalthm} considering the $k$th iteration of the Frobenius map as $\phi_k$.
\end{proof}

Of course, the converse of the above corollary does not hold, since the cohomological dimension of $\aa$ is an invariant of the radical of $\aa$. 

\begin{es}
Notice that Corollary \ref{corcharp} does not hold just assuming $R/\aa$ Cohen-Macaulay: For instance, let $R=\kk[x,y]/(xy)$ and $\aa =(x)\subseteq R$. Note that $R/\aa\cong \kk[y]$ is Cohen-Macaulay, $\height(\aa)=0$ and
\[\ldots \xrightarrow{\cdot y} R \xrightarrow{\cdot x} R \xrightarrow{\cdot y} R \xrightarrow{\cdot x} \aa \rightarrow 0\]
is an infinite minimal free resolution of $\aa$. Considering the homomorphism $R\rightarrow \kk[x]$ mapping $x$ to itself and $y$ to zero and viewing $M=\kk[x]$ as an $R$-module, one has $H_{\aa}^1(M)\cong H_{(x)\kk[x]}^1(\kk[x])\neq 0$. In particular, $\cd(R,\aa)=1>0=\height(\aa)$.
\end{es}

We notice other two consequences of Lemma \ref{generalthm}.

\begin{cor}
Let $A$ be a Noetherian ring and $\Gamma$ a finitely generated commutative monoid. If $\aa\subseteq R=A[\Gamma]$ is a perfect ideal generated by a subset of $\Gamma$, then $\cd(R,\aa)=\height(\aa)$. 
\end{cor}
\begin{proof}
Once again, we want to use Lemma \ref{generalthm}. To this purpose, we look at $\phi_k$ induced by the monoid homomorphisms $\Gamma \xrightarrow{\cdot k}\Gamma$,
which satisfy the assumptions of Lemma \ref{generalthm}.
\end{proof}

\begin{cor}\label{maxmin}
Let $R$ be a Noetherian ring of positive characteristic and $(r_{ij})$ a $m\times n$-matrix with entries in $R$. Let $\aa$ be the ideal generated by the $\min\{m,n\}$-minors of $(r_{ij})$. If $\height(\aa)=|n-m|+1$, then $\cd(\aa)=\height(\aa)$.
\end{cor}
\begin{proof}
Such an ideal $\aa$ is resolved by the Eagon-Northcott complex, so it is perfect. Therefore the conclusion follows by Corollary \ref{corcharp}.
\end{proof}

%

Corollary \ref{maxmin} (and therefore also \ref{corcharp}) does not hold in characteristic $0$, as shown by the following example:

\begin{es}\label{exdet}
Let $\kk$ be a field of characteristic $0$, $(x_{ij})$ a $m\times n$-matrix of indeterminates over $\kk$ and $R=\kk[x_{ij}]$. If $\aa\subseteq R$ is the ideal generated by the $t$-minors of $(x_{ij})$, for $t\leq \min\{m,n\}$, then $\aa$ is a perfect ideal of $\height(\aa)=(m-t+1)(n-t+1)$. However, by \cite{BS}, we have $\cd(R,\aa)=mn-t^2+1$. Therefore,
\[\cd(R,\aa)-\height(\aa)=(m+n-2t)(t-1),\]
that, unless $m=n=t$ or $t=1$, is a positive integer.
\end{es}

\section{Characteristic $0$}

Notice that in Example \ref{exdet} we have $\dim R/\aa = (t-1)(m+n-t+1)$. As one can check, $d=\dim R/\aa$ can be any natural number different from $1$ and $2$. Furthermore, if $d\in\{0,3\}$, then we are forced to be in the special instances in which $\cd(R,\aa)=\height(\aa)$. However, if $d\geq 4$, numbers for which $\cd(R,\aa)>\height(\aa)$ can always be chosen (for example, $t=m=2$ and $n=d-1$, the case in which $\aa$ defines the Segre product $\PP^1\times \PP^{d-2}$ inside $\PP^{2d-3}$). In this section, we try to understand what happens in the remaining cases. Slightly more generally, we wonder if we can deduce $\cd(R,\aa)\leq n-t$ knowing that $\depth(R/\aa)\geq t$, where $t\leq 3$. If $t=0$, then the vanishing theorem of Grothendieck implies $\cd(R/\aa)\leq n$. If $t=1$, then one can show $\cd(R,\aa)\leq n-1$ using the theorem of Hartshorne and Lichtembaum \cite[Theorem 3.1]{Ha}. When $t=2$, one can show that $\cd(R,\aa)\leq n-2$ provided that $R$ is a regular local ring containing a field, exploiting a result of Ogus \cite[Corollary 2.11]{Og} and one of Hartshorne \cite[Proposition 2.1]{hartshorne}. We will point out the proof of this last case.

\begin{prop}\label{dimension 2}
Let $(R,\mm)$ be an $n$-dimensional regular local ring containing a field and $\aa\subseteq R$ an ideal. If $\depth(R/\aa)\geq 2$, then $\cd(R,\aa)\leq n-2$.
\end{prop}
\begin{proof}
Suppose $\kk=R/\mm$ where $\mm$ is the maximal ideal of $R$. If $\chara(\kk)>0$ we already know the result, so we can assume $\chara(\kk)=0$. Take a faithfully flat homomorphism from $(R,\mm)$ to a regular local ring $(S,\nn)$ such that $S/\nn$ is the algebraic closure of $\kk$ (such a thing exists, it is a suitable {\it gonflement} of $R$, see \cite[Chapter IX, Appendice 2.]{Bou}). Faithfully flatness guarantees that $S$ still contains a field, $\depth(R/\aa)=\depth(S/\aa S)$ and $\cd(R,\aa)=\cd(S,\aa S)$. Therefore we can assume that $\kk$ is algebraically closed. 
Again since $\widehat{R}$ is faithfully flat over $R$, we have $\cd(\widehat{R},\aa \widehat{R})=\cd(R,\aa)$ and $\depth(\widehat{R}/\aa\widehat{R})=\depth(R/\aa)$. Thus it is harmless to assume that $R$ is complete, so that $R\cong \kk[[x_1,\ldots ,x_n]]$ by the Cohen structure theorem. By \cite[Proposition 2.1]{hartshorne} $\Spec(R/\aa)\setminus \{\mm/\aa\}$ is connected. So \cite[Corollary 2.11]{Og} yields the conclusion.
\end{proof}
%

For what said till now, it remains to understand the case in which $t=3$: If $\aa$ is an ideal of a regular local ring such that $\depth(R/\aa)\geq 3$, is it true that $\cd(R,\aa)\leq n-3$?

\begin{prop}\label{lyuthm}
Let $(R,\mm)$ be an $n$-dimensional regular local ring containing a field and $\aa\subseteq R$ an ideal. If $\depth(R/\aa)=k$, then $\dim(\Supp(H_{\aa}^{n-i}(R))\leq i-2$ for all $0\leq i<k$.
\end{prop}
\begin{proof}
Given $i<k$, we have to show that $(H_{\aa}^{n-i}(R))_{\mathfrak{p}}=H_{\aa R_{\mathfrak{p}}}^{n-i}(R_{\mathfrak{p}})=0$ for all $\mathfrak{p}\in\Spec R$ such that $\height(\mathfrak{p})\leq n-i+1$. Let us denote by $h$ the height of $\mathfrak{p}$. We can suppose that $h\geq n-i$, because otherwise $H_{\aa R_{\mathfrak{p}}}^{n-i}(R_{\mathfrak{p}})$ is automatically $0$ (since $\dim R_{\mathfrak{p}}=h<n-i$).

First let us assume that $h=n-i$. Since $i<k=\depth(R/\aa)$, $\mathfrak{p}$ is not a minimal prime of $R/\aa$. This implies that $\dim R_{\mathfrak{p}}/\aa R_{\mathfrak{p}} > 0$, which, using the Hartshorne-Lichtembaum theorem, yields $H_{\aa R_{\mathfrak{p}}}^{n-i}(R_{\mathfrak{p}})=0$.

So we can suppose $h=n-i+1$. A theorem of Ischebeck (\cite[Theorem 17.1]{matsu}) yields
\[\Ext_{R}^0(R/\mathfrak{p} ,R/\aa)=\Ext_{R}^1(R/\mathfrak{p} ,R/\aa)=0.\]
This means that $\grade(\mathfrak{p} ,R/\aa) > 1$ and so that $H_{\mathfrak{p}}^0(R/\aa)=H_{\mathfrak{p}}^1(R/\aa)=0$. In particular
\[H_{\mathfrak{p} R_{\mathfrak{p}}}^0(R_{\mathfrak{p}}/\aa R_{\mathfrak{p}})=H_{\mathfrak{p} R_{\mathfrak{p}}}^1(R_{\mathfrak{p}}/\aa R_{\mathfrak{p}})=0,\]
that is $\depth(R_{\mathfrak{p}}/\aa R_{\mathfrak{p}})\geq 2$. Since $R_{\mathfrak{p}}$ is an $(n-i+1)$-dimensional regular local ring, Proposition \ref{dimension 2} yields
\[H_{\aa R_{\mathfrak{p}}}^{n-i}(R_{\mathfrak{p}})=0.\]
This concludes the proof.
\end{proof}

A first consequence of Proposition \ref{lyuthm}, concerns a fact on the Lyubeznik numbers of a local ring. We recall the definition: Let $A$ be a local ring which admits a surjection from an $n$-dimensional regular local ring containing a field. Let $\aa$ be the kernel of the surjection, and $\kk=R/\mm$. In \cite{Ly2}, Lyubeznik proved that the Bass numbers $\lambda_{i,j}(A)=\dim_{\kk} \Ext_R^i(\kk ,H_{\aa}^{n-j}(R))$ depend only on $A$, $i$ and $j$, but neither on $R$ nor on the surjection $R\rightarrow A$. For this reason they are usually called the Lyubeznik numbers of $A$. Furthermore they can been defined for any local ring containing a field, if needed passing to the completion. In \cite{Ly2} it was also showed that $\lambda_{i,j}(A)=0$ whenever $j>d=\dim A$ or $i>j$ and that $\lambda_{d,d}(A)\neq 0$.

\begin{cor}\label{lyucor}
Let $A$ be a local ring containing a field. If $\depth(A)=k$, then $\lambda_{i-1,i}(A)=\lambda_{i,i}(A)=0$ for all $0\leq i<k$.
\end{cor}
\begin{proof}
By \cite[Corollary 3.6]{Ly2} the injective dimension of $H_{\aa}^i(R)$ is bounded above by the dimension of the support of $H_{\aa}^i(R)$, so the statement follows by Proposition \ref{lyuthm}.
\end{proof}

We do not know whether local rings as in Corollary \ref{lyucor} satisfy $\lambda_{i-2,i}(A)=0$ for all $i<k$. Actually this is related to the question we are investigating on. However, we can provide an example for which $k>3$ and $\lambda_{0,3}(A)\neq 0$.

\begin{es}
Let $I$ be the homogeneous ideal of $S=\kk[x_{pq}:p=0,\ldots ,r \ \ q=0,\ldots ,s]$ defining the Segre product $\PP^r\times \PP^s\subseteq \PP^{rs+r+s}$, where $\kk$ is a field of characteristic $0$ and $r>s\geq 1$. Let $\mm$ be the maximal irrelevant of $S$, $R=S_{\mm}$, $\aa =IR$ and $A=R/\aa$. We know from Example \ref{exdet} that, if $n=rs+r+s+1$, then $H_{\aa}^{n-3}(R)\neq 0$. Moreover, if $\mathfrak{p}\in\Spec R$ is not maximal, then $H_{\aa}^{n-3}(R)_{\mathfrak{p}}=0$: In fact, since $\PP^r\times \PP^s$ is smooth, $\aa R_{\mathfrak{p}}\subseteq R$ is generated by a regular sequence of length $n-r-s-1<n-3$. Therefore $\dim(\Supp(H_{\aa}^{n-3}(R))=0$, so $H_{\aa}^{n-3}(R)$ is an injective $R$-module by \cite{Ly2}. Because it is supported at $\mm$, we have $H_{\aa}^{n-3}(R)\cong E^s$ for some $s>0$, where $E$ is the injective hull of $\kk$ (as an $R$-module). Eventually, $\Hom_R(\kk,E^s)\cong\kk^s$, that implies $\lambda_{0,3}(A)\neq 0$.
\end{es}

Before stating the main result of the paper, let us introduce some notation: Let $X$ be a projective scheme over a field $\kk$ of characteristic $0$. By $H_{DR}^i(X)$ we mean algebraic de Rham cohomology, as defined in \cite{hartder}. If $\kk=\CC$, we denote by $X_h$ the analytic space associated to $X$. By \cite[Chapter IV, Theorem 1.1]{hartder}, we have $H^i(X_h,\CC)\cong H_{DR}^i(X)$, where by $H^i(X_h,\CC)$ we mean singular cohomology with coefficients in $\CC$. 

\begin{thm}\label{mainthm}
Let $\kk$ be a field of any characteristic and $R=\kk[x_1,\ldots ,x_n]$. If $\aa\subseteq R$ is a homogeneous ideal such that $\depth(R/\aa)\geq 3$, then $\cd(R,\aa)\leq n-3$. In particular, if $R/\aa$ is a $3$-dimensional Cohen-Macaulay ring, then $\cd(R,\aa)=n-3$.
\end{thm}
\begin{proof}
Of course we need to show this only when the characteristic of $\kk$ is $0$, since in positive characteristic it is true by the result of Peskine-Szpiro. The properties in the hypothesis and in the thesis are preserved under flat extensions. So, since $\aa$ is finitely generated, we can assume that $\QQ\subseteq \kk\subseteq \CC$, and eventually that $\kk=\CC$. That $\cd(R,\aa)\leq n-3$ is equivalent to say that $H^i(\PP^{n-1}\setminus X,\F)=0$ for any coherent sheaf $\F$ on $X=\V_+(\aa)$ and for all $i\geq n-3$. This, by \cite[Theorem 4.4]{Og}, is equivalent to say that $H_{DR}^0(X)\cong H_{DR}^0(\PP^{n-1})$, $H_{DR}^1(X)\cong H_{DR}^1(\PP^{n-1})$ and the de Rham depth of $X$ is at least $2$. The last condition is in turn equivalent to $\Supp(H_{\aa}^i(R))\subseteq \{\mm\}$ for all $i\geq n-2$ where $\mm$ is the maximal irrelevant ideal (see the proof of \cite[Theorem 4.1]{Og}); this is true by Proposition \ref{lyuthm}. The condition $H_{DR}^0(X)\cong H_{DR}^0(\PP^{n-1})$ means that $X$ is connected, which is the case because $\depth(R/\aa)\geq 3$, see \cite[Proposition 2.1]{hartshorne}. So we have to show that $H_{DR}^1(X)\cong H_{DR}^1(\PP^{n-1})=0$. Since $H_{DR}^1(X)\cong H^1(X_h,\CC)$, by the universal coefficient theorem it is enough to show that $H^1(X_h,\ZZ)$ is zero. Let us consider the morphisms of sheaves $\ZZ_{X_h}\rightarrow \O_{X_h}\rightarrow (\O_{{X_h}})_{\red}$ (here $\ZZ_{X_h}$ denotes the locally constant sheaf on $X_h$ associated to $\ZZ$). By the exponential sequence we know that the composition 
\[H^1(X_h,\ZZ_{X_h})\rightarrow H^1(X_h,\O_{X_h})\rightarrow H^1(X_h,(\O_{X_h})_{\red})\]
is injective, so $H^1(X_h,\ZZ_{X_h})\rightarrow H^1(X_h,\O_{X_h})$ has to be injective too (actually the exponential sequence makes sense also for non reduced schemes, however we preferred to follow a different route because the lack of references). By \cite{GAGA} $H^1(X_h,\O_{X_h})\cong H^1(X,\O_{X})$. Furthermore $H^1(X,\O_X)\cong H_{\mm}^2(R/\aa)_0$, and the last vector space is zero because $\depth(R/\aa)\geq 3$. So $H^1(X_h,\ZZ)\cong H^1(X_h,\ZZ_{X_h})$ has to be zero.
%
%
\end{proof}

\begin{os}
Under the assumption that $X$ is smooth, where the situation is considerably simpler, we got the conclusion of Theorem \ref{mainthm} in \cite{Va}. 
Even in the smooth case, however, the above proof cannot be repeated to show $\depth(R/\aa)\geq 4\implies \cd(R,\aa)\leq n-4$. The point is that 
\[H^2(X_h,\CC_{X_h})\cong H^2(X_h,\O_{X_h})\oplus H^1(X_h,\Omega^1_{X_h}) \oplus H^0(X_h,\Omega^2_{X_h}),\] 
so, by the presence of the middle term $H^1(X_h,\Omega^1_{X_h})$, $H_{\mm}^3(R/\aa)=0$ does not imply $H^2(X_h,\CC_{X_h})=0$.
\end{os}

\begin{os}\label{setdepth}
In \cite[Theorem 3.3]{SW}, Singh and Walther used characteristic $p$ methods to prove the following. If $E\subseteq \PP^2$ is an elliptic curve defined over $\ZZ$, $\kk$ is a field of characteristic $0$ and $\aa\subseteq R=\kk[x_0,\ldots ,x_5]$ is the defining ideal of $E\times \PP^1$, then $R/\bb$ is not Cohen-Macaulay for all homogeneous ideals $\bb$ with the same radical of $\aa$. Theorem \ref{mainthm} yields immediately a much more general fact: Let $C$ be a smooth curve of genus at least $1$ and $Y$ be any projective scheme, both over a field $\kk$ of characteristic $0$. Let $\aa\subseteq R= \kk[x_1,\ldots ,x_N]$ be the defining ideal of $X=C\times Y\subseteq \PP^{N-1}$. For all homogeneous $\bb\subseteq R$ such that $\sqrt{\bb}=\sqrt{\aa}$, then $\depth(R/\bb)\leq 2$. To show this, we claim that $\cd(R,\aa)\geq N- 2$, so that, since the cohomological dimension is independent on the radical, Theorem \ref{mainthm} would give the conclusion. To this purpose we can assume $\kk=\CC$. GAGA and Hodge decomposition imply that $H^1(C_h,\CC)\neq 0$: So $H^1(X_h,\CC)\neq 0$ by the Kunneth formula, thus $\cd(R,\aa)\geq N-2$ by a result of Hartshorne, see \cite[Theorem 7.4, p. 148]{hartshorne4}. 
\end{os}

The above remark shows how Theorem \ref{mainthm} can be used to produce ideals which are not set-theoretically Cohen-Macaulay. The following Proposition gives further such examples.

\begin{prop}\label{hodge}
Let $\kk$ be a field of characteristic $0$, $\aa\subseteq R=\kk[x_1,\ldots ,x_n]$ a graded ideal and $\mm=(x_1, \ldots ,x_n)$ the maximal irrelevant. Assume that $\bb=\sqrt{\aa}$ is such that $X=\Proj(R/\bb)$ is smooth. Then
\[\dim_{\kk}H_{\mm}^i(R/\aa)_0\geq \dim_{\kk}H_{\mm}^i(R/\bb)_0 \ \ \forall i\geq 0.\]
\end{prop}
\begin{proof}
For $i=0$ it is trivial. By the exact sequences of graded $R$-modules:
\begin{eqnarray*}
0\rightarrow H_{\mm}^0(R/\aa)\rightarrow R/\aa \rightarrow \bigoplus_{i\in\NN}H^0(X,\O_X(i))\rightarrow H_{\mm}^1(R/\aa)\rightarrow 0 \\
0\rightarrow H_{\mm}^0(R/\bb)\rightarrow R/\bb \rightarrow \bigoplus_{i\in\NN}H^0(X,(\O_X)_{\red}(i))\rightarrow H_{\mm}^1(R/\bb)\rightarrow 0
\end{eqnarray*}
the dimension of the $\kk$-vector spaces $H_{\mm}^1(R/\aa)_0$ and $H_{\mm}^1(R/\bb)_0$ are, respectively, $\dim_{\kk}H^0(X,\O_X)-1$ and $\dim_{\kk}H^0(X,(\O_{X})_{\red})-1$, and obviously $\dim_{\kk}H^0(X,\O_{X})\geq\dim_{\kk}H^0(X,(\O_{X})_{\red})$. Note that so far we did not use the smoothness of $X_{\red}$. For $i\geq 2$ we have to. As usual, it is harmless to assume $\kk=\CC$. Let us remind the isomorphisms of $\CC$-vector spaces $H_{\mm}^i(R/\aa)_0\cong H^{i-1}(X,\O_X)$ and $H_{\mm}^i(R/\bb)_0\cong H^{i-1}(X,(\O_X)_{\red})$. Let us consider the natural maps of sheaves
\[\CC_{X_h}\rightarrow \O_{X_h} \rightarrow (\O_{X_h})_{\red}.\]
These yield maps of $\CC$-vector spaces
\[H^i(X_h,\CC_{X_h})\xrightarrow{\alpha} H^i(X_h,\O_{X_h}) \xrightarrow{\beta} H^i(X_h,(\O_{X_h})_{\red}).\]
The composition of these homomorphisms is surjective: Indeed by Hodge theory $H^i(X_h,(\O_{X_h})_{\red})$ is the space of harmonic $(i,0)$ forms, and Hodge decomposition  says that $\beta\alpha$ maps a harmonic $i$-form to its $(i,0)$ component (see the book of Arapura \cite{Ar} for unexplained terminology). Therefore $H^i(X_h,\O_{X_h}) \rightarrow H^i(X_h,(\O_{X_h})_{\red})$ is surjective. By \cite{GAGA} there are isomorphism of $\CC$-vector spaces $H^i(X,\O_X)\cong H^i(X_h,\O_{X_h})$ and $H^i(X,(\O_{X})_{\red})\cong H^i(X_h,(\O_{X_h})_{\red})$, thus we conclude.
\end{proof}

The smoothness assumption in Proposition \ref{hodge} is necessary, how demonstrates the following example due to Aldo Conca.

\begin{es}
Let $R=\QQ[x_1,\ldots ,x_6]$ and
\[\aa=(x_2x_4+x_3x_6, \ \ x_3^2-x_4^2, \ \ x_1^2+x_4x_5).\]
It turns out that $\aa$ is a complete intersection. Particularly, $H_{\mm}^2(R/\aa)_0=0$ because $R/\aa$ is a $3$-dimensional Cohen-Macaulay ring. However, using Macaulay2 \cite{mac}, one sees that $H_{\mm}^2(R/\bb)_0\cong \Ext_S^4(R/\bb,R)_{-6}$ is a $1$-dimensional $\QQ$-vector space, where
\[\bb=\sqrt{\aa}=(x_2x_4+x_3x_6, \ \ x_3^2-x_4^2, \ \ x_1^2+x_4x_5, \ \ x_2x_3+x_4x_6, \ \ x_1(x_2^2-x_6^2)).\]
Notice that $\Proj(R/\bb)$ is not smooth. Indeed, it is not even irreducible, though connected; one can check that the minimal prime ideals of $\bb$ are:
\[\mathfrak{p}_1=(x_3+x_4, \ x_2-x_6, \ x_1^2+x_4x_5), \ \ \ \mathfrak{p}_2=(x_1, \ x_3, \ x_4), \ \ \ \mathfrak{p}_3=(x_3-x_4, \ x_2+x_6, \ x_1^2+x_4x_5).\]
Notice that the same example works in every field $\kk$ of characteristic $0$. Indeed, setting $R_{\kk}=R\otimes_{\QQ}\kk$, the $\mathfrak{p}_iR_{\kk}$'s keep on being prime ideals, $\sqrt{\aa R_{\kk}}=\bb R_{\kk}$, $H_{\mm R_{\kk}}^2(R_{\kk}/\aa R_{\kk})_0=H_{\mm}^2(R/\aa)_0\otimes_{\QQ}\kk=0$ and 
\[\dim_{\kk}H_{\mm R_{\kk}}^2(R_{\kk}/\sqrt{\aa R_{\kk}})_0=\dim_{\kk}H_{\mm}^2(R/\bb)_0\otimes_{\QQ}\kk=1.\]
\end{es}

\end{document}